\documentclass{article}%
\usepackage{graphicx}
\usepackage{amsmath,amssymb,amsfonts}%
\usepackage{theorem}%
\usepackage{color}%
\usepackage{hyperref}
\usepackage[font={small, it}]{caption}
\usepackage{caption}
\usepackage{subcaption}
\usepackage{float}
\usepackage{tikz-cd}
\tikzcdset{scale cd/.style={every label/.append style={scale=#1},
    cells={nodes={scale=#1}}}}
\usepackage{forest}
\usepackage{tikz-qtree}
\usetikzlibrary{trees}    
\theoremstyle{change}%
\usepackage{cite}

\newtheorem{definition}{Definition:}[section]%
\newtheorem{proposition}[definition]{Proposition:}%
\newtheorem{theorem}[definition]{Theorem:}%
\newtheorem{lemma}[definition]{Lemma:}%
{\theorembodyfont{\rmfamily}\newtheorem{remark}[definition]{Remark:}}%
{\theorembodyfont{\rmfamily}}%

\newenvironment{proof}
{{\bf Proof:}}
{\qquad \hspace*{\fill} $\Box$}%

\newcommand{\inner}{\operatorname{int}}%
\newcommand{\rme}{\mathrm{e}}%

\newcommand{\CC}{\mathcal{C}}%
\newcommand{\XC}{\mathcal{X}}%
\newcommand{\DC}{\mathcal{D}}%

\newcommand{\HB}{\mathbb{H}}%
\newcommand{\R}{\mathbb{R}}%

\begin{document}
	
	\title{One-input linear control systems on the homogeneous spaces of the Heisenberg group - The singular case}

\author{Adriano Da Silva\thanks{Supported by Proyecto UTA Mayor Nº 4768-23} \\
		Departamento de Matem\'atica,\\Universidad de Tarapac\'a - Sede Iquique, Chile.
		\and
		Okan Duman\; and\; Ey\"up Kizil \\
		Department of Mathematics \\
		Yildiz Technical University - Istanbul, Turkey.\\
}
\date{\today }
\maketitle

\begin{abstract}
The controllability issue of control-affine systems on smooth manifolds is one of the main problems in the theory, and it is recently known \cite{Jouan} that it might be connected to that of a particular class of systems called linear control systems on (a homogeneous manifold of) a Lie group. Note that it may become very complicated to establish the controllability property of systems evolving on homogeneous spaces of Lie groups whose dynamics are induced by those of systems in the Lie group under consideration. In fact, even in low-dimensional certain homogeneous spaces, this is quite a challenging task, and for this reason, we have classified in \cite{AEO} as a first goal all linear control systems on the homogeneous spaces of the 3-dimensional Heisenberg group $\HB$ through its closed subgroups $L$ and, in particular, the controllability and the control sets have been studied for one of the homogeneous spaces $L\backslash \HB$.

In this paper, we study the controllability and control sets of the induced linear control systems in the homogeneous spaces left. In particular, we focus on the singularity of the induced drift vector fields that results in many cases and subcases to reveal control sets after quite a technical analysis. We give some nice illustrations to better understand what is going on geometrically.
\end{abstract}

	

\section{Introduction}

The study of control systems has a rich history, with a particular focus on
linear control systems on $\mathbb{R}^{n}$, which plays a crucial role in
understanding various real-world applications (see, for example, \cite{FW,
GL, KS, Markus}). In essence, a control system can be described as a
dynamical system defined on a state space (a differentiable manifold) with
dynamics governed by a family of differential equations parameterized
through control functions. A natural extension to the matrix Lie groups of a
linear control system on Euclidean space is introduced in \cite{Markus} by
L. Marcus and further extended to arbitrary connected Lie groups in \cite%
{Ayala} by V. Ayala and J. Tirao. In the ensuing years, many authors have
addressed outstanding problems in control theory for this particular class
of systems. These problems include essential questions such as
controllability, observability, and optimization, among others, with notable
contributions in the theory, see \cite{VAAA, VA1, VA2, VA3, VAP, AEO} and
the references therein. One of the key insights to understanding the dynamical
properties of a control system is to study control sets in both topological
and/or algebraic sense. These sets provide the maximum regions in the state
space where approximate controllability occurs.

In this paper, we pick as a state space the well-known Heisenberg Lie group $\mathbb{H}$ of dimension three and consider its one-dimensional closed subgroups $L$ to get two-dimensional homogeneous spaces $L\backslash\mathbb{H}$. This is because, even in the context of low-dimensional groups, the properties of control sets can differ significantly when considering dynamics in Lie groups and their homogeneous spaces (see \cite{VAAA} and \cite{VA1}). It also becomes a very challenging task to analyze the controllability property and to characterize the topological properties of control sets for all these systems. On the other hand, we know from a recent result of P. Jouan in \cite{Jouan} that, in a more general setting, a control-affine system on a manifold can be connected to a linear system on either a Lie group or a homogeneous space. Because of these important connections, we have classified in \cite{AEO} all linear control systems on the homogeneous spaces $L\backslash\mathbb{H}$, and our main purpose here is now to give a very detailed analysis of the controllability properties of these systems. In particular, we show the control behavior of the dynamics evolving on simply connected homogeneous (state) spaces of dimension two for the class of singular systems. By "singular," we mean the singular determinant of the associated derivation matrix for the (induced) drift.

The paper is divided into three main sections. Section \ref{preli} introduces some generalities in the control theory to enhance comprehension of the rest of the manuscript. We also expose explicitly in this chapter the dynamics of $\HB$ and the structure of the corresponding homogeneous spaces equipped with its induced dynamics. In our previous work \cite{AEO}, we have determined 1-dimensional closed subgroups of $\mathbb{H}$ and classified normal and non-normal ones. Our main focus there was to understand the structure of homogeneous spaces obtained by non-normal subgroups, since the treatment of the normal case has already been studied in some recent works (refer to \cite{Ayala, VA1, VAP, VA2, VA3}). Note that the non-normal subgroups of $\mathbb{H}$ were initially determined as (i) $L=\mathbb{R}\mathbf{e}_{1}\times \mathbb{Z}$ and (ii) $L=\mathbb{R}\mathbf{e}_{1}\times \{0\}$, and we have thoroughly studied the structure of all LCSs on the corresponding homogeneous spaces $L\backslash \mathbb{H}\simeq \mathbb{R}\times \mathbb{T}$ and $L\backslash \mathbb{H}\simeq \mathbb{R}\times \mathbb{R}$, respectively. In particular, any induced control system $\Sigma$ on $L\backslash \mathbb{H}$ is equivalent to one of the systems indicated in the diagram (\ref{diagram1}) below. Since the controllability of such a system on the homogeneous space $\mathbb{R}\times \mathbb{T}$ was already studied in \cite{AEO}, our current interest in this article is the controllability of the induced control systems on the other homogeneous space $\mathbb{R}\times \mathbb{R}$. We have to emphasize that the latter case presents a much more intricate and complex structure, which makes this paper both challenging and interesting enough. The controllability issue of these linear control systems in a very detailed manner will be studied in Section \ref{controlsec} by means of a quite technical case-by-case analysis. Finally, in this section, the structures of the control sets for the systems under consideration are explicitly presented for each case and its sub-case whenever they exist.

\section{Preliminaries}\label{preli}
In this section, we introduce very basic knowledge about dynamical systems
for a better understanding of the rest of the paper. We start with control
affine systems on smooth manifolds and quote some results that are necessary
for later references.

\subsection{Control-affine systems}
Let $M$ be a finite-dimensional smooth manifold, and let $\mathbb{R}^{m}$
denote the $m$-dimensional Euclidean space. Given a compact convex subset,
$\Omega \subset \mathbb{R}^{m}$ satisfying $0\in $ int $\Omega $, we mean by
a control-affine system evolving on $M$ the following (parametrized) family
of ordinary differential equations
\begin{equation*}
\Sigma _{M}:\quad \dot{x}(t)=f_{0}(x(t))+\sum_{j=1}^{m}\omega_{j}(t)f_{j}(x(t)),%
\quad \omega\in \mathcal{U},
\end{equation*}%
where $f_{0},f_{1},\ldots ,f_{m}$ are smooth vector fields defined on $M$
and the control parameter $\omega=\left( \omega_{1},\ldots ,\omega_{m}\right) $ belongs to
the set $\mathcal{U}$ of the piecewise constant functions such that $\omega(t)\in
\Omega $.

For an initial state $x\in M$ and $\omega\in \mathcal{U}$, the solution of $%
\Sigma _{M}$ is the unique absolutely continuous curve $\tau\mapsto \varphi (\tau,x,\omega)
$ on $M$ satisfying $\varphi (0,x,\omega)=x$. Associated to $\Sigma _{M}$ we have
for a given $x\in M$, the positive and negative orbits at $x$ as follows:
\begin{eqnarray*}
\mathcal{O}^{+}(x) &=&\{\varphi (\tau,x,\omega):\tau\geq 0,u\in \mathcal{U}\}, \\
\mathcal{O}^{-}(x) &=&\{\varphi (-\tau,x,\omega):\tau\geq 0,u\in \mathcal{U}\}\text{.}
\end{eqnarray*}%
\begin{definition}
We say that a control affine system $\Sigma _{M}$ on $M$ satisfies the \emph{%
Lie algebra rank condition} (abrev. LARC) if $\mathcal{L}(x)=T_{x}M$ for all
$x\in M,$ where $\mathcal{L}$ means the smallest Lie algebra of vector
fields containing $\Sigma _{M}$. The system $\Sigma _{M}$ is said to \emph{%
controllable} if $M=\mathcal{O}^{+}(x)$ for all $x\in M$.
\end{definition}
Note that controllability is quite a strong property for a control system to
satisfy so that an alternative way for controllability is to study the
maximal subsets of $M$ where it occurs. Such subsets are known in the theory
as control sets, which we define next.

\begin{definition}
A nonempty set $\mathcal{C}\subset M$ is a control set of $\Sigma _{M}$ if it is
maximal, w.r.t. set inclusion, with the following properties:

\begin{enumerate}
\item $\forall x\in \mathcal{C}$, there exists a control $\omega\in \mathcal{U}$
such that $\varphi \left( \mathbb{R}^{+},x,\omega\right) \subset \mathcal{C}$;

\item It holds that $\mathcal{C}\subset \mathrm{cl}~\mathcal{O}^{+}(x)$ for
all $x\in \mathcal{C}$.
\end{enumerate}
\end{definition}
It is known that \cite[Proposition 3.2.4]{FW} any subset $\mathcal{C}$ of $M$
with a nonempty interior that is maximal concerning the above property 2 is a
control set. To understand the dynamics of a control system, it is
essential to capture the topological, geometric, and/or algebraic properties
of its control sets. For instance, they allow us to obtain many dynamical
properties of the system, such as equilibrium and recurrence points,
periodic and bounded orbits, etc. Furthermore, the exact controllability is
satisfied in its interior, which means that points can be steered into each
other by a solution of the system in positive time.

Below, we introduce conjugations between control-affine systems. Such
concepts help us simplify their dynamical analysis by changing the
coordinates of the manifold $M$. Let $N$ be another smooth manifold and

\begin{equation*}
\Sigma _{N}:\quad \dot{y}(t)=g_{0}(y(t))+\sum_{j=1}^{m}\omega_{j}(t)g_{j}(y(t)),%
\quad \omega\in \mathcal{U},
\end{equation*}%
a control-affine system on $N$.

\begin{definition}\label{defi:conjugado}
If $\psi :M\rightarrow N$ is a smooth map, we say that a vector field $X$ on
$M$ and a vector field $Y$ on $N$ are $\psi $-conjugated (sometimes said $%
\psi $-related) if $d\psi \circ X=Y\circ \psi $. In particular, we say that $%
\Sigma _{M}$ and $\Sigma _{N}$ are $\psi $-conjugated if $d\psi \circ
f_{j}=g_{j}\circ \psi $ for each $j\in \{0,1,2,\ldots ,m\}$. In case $\psi $
is a diffeomorphism, $\Sigma _{M}$ and $\Sigma _{N}$ are called equivalent
systems.
\end{definition}

Several properties of equivalent systems are preserved, such as
controllability, topological properties of positive and negative orbits, and
control sets. In this direction, the next result
relates control sets of conjugated systems. Since its proof is standard, we will omit it.

\begin{proposition}
\label{conjugation}
    Let $\Sigma_M$ and $\Sigma_N$ be $\psi$-conjugated systems. It holds:
\begin{enumerate}
    \item[1.] If $\CC_M$ is a control set of $\Sigma_M$, there exists a control set $\CC_N$ of $\Sigma_N$ such that $\psi(\CC_M)\subset \CC_N$;
    \item[2.] If for some $y_0\in \inner \CC_N$ it holds that $\psi^{-1}(y_0)\subset\inner \CC_M$, then $\CC_M=\psi^{-1}(\CC_N)$.
\end{enumerate}
\end{proposition}
Since we are going to examine control affine systems on the Heisenberg group
(and consequently, its associated homogeneous space), it is convenient to
expose fundamental definitions and facts concerning Lie groups and their
associated Lie algebras.

\begin{definition}
A vector field $\mathcal{X}$ on a connected Lie group $G$ is linear if its
flow $\{\varphi _{t}\}_{t\in \mathbb{R}}$ is a 1-parameter subgroup of $%
Aut(G)$, the group of all automorphisms of $G$.
\end{definition}

It is well known that a linear vector field on a connected Lie group $G$ is
complete, and one can always associate with such a vector field a derivation $%
\mathcal{D}=-ad(\mathcal{X})$ of the corresponding Lie algebra $\mathfrak{g}$
of $G$. Recall that by a derivation $\mathcal{D}$ of a Lie algebra $%
\mathfrak{g}$ we mean a linear map on $\mathfrak{g}$ such that $\mathcal{D}%
[X,Y]=[\mathcal{D}X,Y]+[X,\mathcal{D}Y]$ for every $X,Y\in \mathfrak{g}$.
Therefore, it holds that $(d\varphi _{t})_{e}=e^{t\mathcal{D}},$ $\forall
t\in \mathbb{R}$, where $\varphi _{t}$ stands for the flow. In particular,
from the following diagram
\[ \begin{tikzcd}[row sep = large, column sep = large]
\mathfrak{g}~~ \arrow{r}{(d\varphi _{t})_{e}} \arrow[swap]{d}{\exp{}} & ~~ \mathfrak{g} \arrow{d}{\exp{}} \\%
G~~\arrow[ur, phantom, "\scalebox{1.5}{$\circlearrowleft$}" description]  \arrow{r}{\varphi _{t}}& ~~ G
\end{tikzcd}
\]
we have that $\varphi _{t}(\exp Y)=\exp (e^{t\DC}Y)$
for every $Y\in \mathfrak{g}$ and $t\in \mathbb{R}$.
\begin{remark}
Any linear vector field defines a derivation, but the converse is true only
for simply connected Lie groups. Furthermore, in the case where $G$ is a
connected and simply connected nilpotent Lie group, the exponential map
serves as a diffeomorphism. This implies that, especially for a given
derivation, it becomes possible to explicitly compute the drift $\mathcal{X}$
through the above diagram using the logarithmic map $\log (g)=Y$ where $g\in
G$.
\end{remark}

\subsection{LCSs on the Heisenberg group $\HB$}
In this section, we mention some results on the LCSs, both on the Heisenberg
group and on the homogeneous spaces generated by its one-dimensional
subgroups. Rather than defining the Heisenberg group $\mathbb{H}$ as the set
of upper triangular matrices with only 1s in the main diagonal, we choose to
interpret it as the Cartesian product $\mathbb{R}^{2}\times \mathbb{R}$.
This perspective enables us to express key elements in a more effective way,
such as group multiplication, invariant and linear vector fields, and their
Lie brackets. The Heisenberg group $\mathbb{H}=\mathbb{R}^{2}\times \mathbb{R%
}$ has a product given by
\begin{equation*}
(\mathbf{v}_{1},z_{1})\ast (\mathbf{v}_{2},z_{2})=\left(\mathbf{v}_{1}+%
\mathbf{v}_{2},z_{1}+z_{2}+\frac{1}{2}\left\langle \mathbf{v}_{1}, \theta%
\mathbf{v}_{2}\right\rangle\right),
\end{equation*}%
where $\left\langle \cdot ,\cdot \right\rangle $ stands for the standard
inner product in $\mathbb{R}^{2}$ and $\theta$ stands for the
counter-clockwise rotation of $\frac{\pi}{2}$-degrees, and $\mathbf{{v}_{1},
\mathbf{v}_{2}} \in \mathbb{R}^2, z_1, z_2\in \mathbb{R}.$

The Lie algebra $\mathfrak{h}$ $=\mathbb{R}^{2}\times \mathbb{R}$ of $%
\mathbb{H}$ is equipped with the Lie bracket
\begin{equation*}
\lbrack (\mathbf{\zeta }_{1},\alpha _{1}),(\mathbf{\zeta }_{2},\alpha
_{2})]=(\mathbf{0},\left\langle \mathbf{\zeta }_{1}, \theta\mathbf{\zeta }%
_{2}\right\rangle ),
\end{equation*}
where ${\zeta_{1}},{\zeta_{2}} \in {\mathbb{R}^2}$ and ${\alpha_1},{\alpha_2}
\in \mathbb{R}$.
Following our previous setup, we state in the next elementary proposition the explicit expression of a
typical derivation $\mathcal{D}$ of $\mathfrak{h}$ and an automorphism of $%
\mathbb{H}$ both in matrix form w.r.t. the standard basis. Recall that $%
\mathrm{Gl}(2)$ denotes the Lie group of $2\times 2$ invertible matrices
with Lie algebra $\mathfrak{gl}(2)$.
\begin{proposition}
Let $P\in \mathrm{Gl}(2)$, $A\in \mathfrak{gl}(2)$, and $\eta \in \mathbb{R}^{2}$. Elements of the Lie algebra $\mathrm{Der}(\mathfrak{h})$ of derivations of $\mathfrak{h}$ and the Lie group $\mathrm{Aut}(\HB)$ of automorphisms of $\HB$ in matrix form are given by:
   \[
  \mathcal{D}= \left(\begin{array}{cc}
 		A & \mathbf{0}\\ \eta^{\top} & \mathrm{tr}A
 	\end{array}\right)\in \mathrm{Der}(\mathfrak{h})
 \qquad  \text{and}\qquad \quad 
  \Theta= \left( 
 \begin{array}{cc}
 P & \mathbf{0} \\ 
 \eta^{\top}  & \det P
 \end{array}\right)\in\mathrm{Aut}(\HB).
 \]
\end{proposition}

Given the fact that the Lie algebra $\mathfrak{h}$ of $\mathbb{H}$ is
regarded as the set of left-invariant vector fields on $\mathbb{H}$, we
present a usual expression for such a vector field that is suitable for our
current context. Additionally, given a derivation $\mathcal{D}$ of $%
\mathfrak{h}$ as described above, one can easily associate $\mathcal{D}$
with the linear vector field $\mathcal{X}$ on $\mathbb{H}$ using the
relation $[B, \mathcal{X}] = \mathcal{D} B$ for all $B \in \mathfrak{h}$.
Therefore, when we refer to a one-input linear control system on $\mathbb{H}$%
, we are describing a system in the following form:
\begin{equation*}
\Sigma _{\mathbb{H}}:\quad \dot{(\mathbf{v},z)}=\mathcal{X}(\mathbf{v}%
,z)+\omega B(\mathbf{v},z),
\end{equation*}
where $\omega \in \Omega $, $\mathcal{X}(\mathbf{v}%
,z)=(A\mathbf{v},\langle \eta ,\mathbf{v%
}\rangle +z~\mathrm{tr}A)$ is a linear vector field, and $B(\mathbf{v}%
,z)=(\zeta ,\alpha +%
\frac{1}{2}\langle \mathbf{v}, \theta\zeta \rangle)$ is a left-invariant
vector field with $(\mathbf{v},z)\in \mathbb{H}$ and $(\zeta ,\alpha)\in%
\mathfrak{h}$. Consequently, the flow $\varphi _{\tau}$ induced by $\mathcal{X}$
is as follows:
\begin{equation*}
\varphi _{\tau}(\mathbf{v},z)=\left( e^{\tau A}\mathbf{v},\left\langle \mathrm{e}%
^{\tau\cdot\mathrm{tr}A}\mathbf{\Lambda }_{\tau}^{(A-\mathrm{tr}A\cdot I_2)}\eta ,%
\mathbf{v}\right\rangle +z\mathrm{e}^{\tau\cdot\mathrm{tr}A}\right),
\end{equation*}
where $I_2$ stands for the $2\times 2$ identity matrix and $\mathbf{\Lambda }%
^{B}:\mathbb{R}\times \mathbb{R}^{2}$ is an operator defined by $\mathbf{%
\Lambda }_{\tau}^{B}\eta =\int_{0}^{\tau}e^{s B^{\top}}\eta ds$ for all $B\in
\mathfrak{gl}(2)$.
\begin{remark}
   Given that a linear vector field on $\HB$ is uniquely determined by $\eta \in \mathbb{R}^2$ and $A \in \mathfrak{gl}(2)$, we will commonly use the notation $\XC = (\eta, A)$ to denote such a vector.
\end{remark}

\subsection{LCSs on the homogeneous spaces of $\HB$}

In this section, we recall some results and useful facts about the homogeneous spaces of the Heisenberg group $\HB $.
\begin{definition}
A one-input LCS on the homogeneous space $L\backslash \mathbb{H}$ is the
following control-affine system:
\begin{equation}
\Sigma _{L\backslash \mathbb{H}}:\quad \dot{p}=f_{0}(p)+\omega f_{1}(p),
\label{Eq:contsyshomogeneous}
\end{equation}%
with $\omega\in \Omega ,p\in L\backslash \mathbb{H}$ and $f_{0}$ and $f_{1}$ are
vector fields on $L\backslash \mathbb{H}$ satisfying
\begin{equation*}
d\pi \circ \mathcal{X}=f_{0}\circ \pi \hspace{0.5cm}\text{ and }\hspace{0.5cm%
}d\pi \circ B=f_{1}\circ \pi,
\end{equation*}%
where $\mathcal{X}$ is a linear vector field and $B$ is a left-invariant vector
field, with $\pi :\mathbb{H}\rightarrow L\backslash \mathbb{H}$ the standard
canonical projection.
\end{definition}

Based on the definition in Definition \ref{defi:conjugado}, it follows that
a LCS on the homogeneous space $L\backslash \mathbb{H}$ is $\pi $-conjugated
to an LCS on $\mathbb{H}$. Furthermore, as established in \cite[Proposition 4%
]{Jouan}, the vector field $d\pi \circ \mathcal{X}$ is well-defined on $%
L\backslash \mathbb{H}$ if and only if $L$ is invariant under the flow $%
\varphi _{t}$ of $\mathcal{X}$ indicating $\varphi _{t}(L)=L$ for every $%
t\in \mathbb{R}$. As can be seen, one needs to find the possible $\varphi
_{t}$-invariant closed subgroups of $\mathbb{H}$ in order to classify all
possible LCSs on the homogeneous spaces of $\mathbb{H}$.

\begin{remark}
Consider $\Sigma _{L\backslash \mathbb{H}}$, a LCS on $L\backslash \mathbb{H}
$ defined as in (\ref{Eq:contsyshomogeneous}). In this context, the vector
fields $\mathcal{X}$ and $B$ are $\pi $-conjugated with the vector fields $%
f_{0}$ and $f_{1}$, respectively. Suppose that there is a $\psi \in \mathrm{
Aut}(\mathbb{H})$ such that $\widehat{L}=\psi (L)$ is one of the invariant
subgroups w.r.t. the flow of the linear vector field $\mathcal{X}$.
Furthermore, let $\widehat{\mathcal{X}}$ and $\hat{B}$ represent the vector
fields satisfying
\begin{equation*}
d\psi \circ \widehat{\mathcal{X}}=\mathcal{X}\circ \psi \hspace{0.5cm}%
\mbox{
and }\hspace{0.5cm}d\psi \circ \widehat{B}=B\circ \psi \text{.}
\end{equation*}%
The invariance of $L$ under the flow of $\mathcal{X}$ implies that $\widehat{%
L}$ is similarly invariant under the flow of $\widehat{\mathcal{X}}$.
Consequently, we obtain well-defined vector fields $\widehat{f}_{0}$ and $%
\widehat{f}_{1}$ on $\widehat{L}\backslash \mathbb{H}$ determined in the
following way
\begin{equation*}
\widehat{f_{0}}\circ \widehat{\pi }=d\widehat{\pi }\circ \widehat{\mathcal{X}%
}\hspace{0.5cm}\mbox{ and }\hspace{0.5cm}\widehat{f_{1}}\circ \widehat{\pi }%
=d\widehat{\pi }\circ \widehat{B},
\end{equation*}%
where $\widehat{\pi }:\mathbb{H}\rightarrow \widehat{L}\backslash \mathbb{H}$
is the canonical projection. Since the map $\widehat{\psi }:L\backslash
\mathbb{H}\rightarrow \widehat{L}\backslash \mathbb{H}$ defined by the
relation $\widehat{\psi }\circ \pi =\widehat{\pi }\circ \psi $ is a
diffeomorphism, the fact that
\begin{equation*}
d\widehat{\psi }\circ f_{0}=\widehat{f}_{0}\circ \widehat{\psi }\hspace{0.5cm%
}\mbox{ and }\hspace{0.5cm}d\widehat{\psi }\circ f_{1}=\widehat{f}_{1}\circ
\widehat{\psi },
\end{equation*}%
give us that $\Sigma _{L\backslash \mathbb{H}}$ is equivalent to the LCS on $%
\widehat{L}\backslash \mathbb{H}$ given by
\begin{equation*}
\Sigma _{\widehat{L}\backslash \mathbb{H}}:\quad \dot{q}=\widehat{f}_{0}(q)+u%
\widehat{f}_{1}(q)\text{.}
\end{equation*}
\end{remark}

\begin{proposition}\label{MainPro1}\cite[Proposition 2]{AEO}
Up to isomorphisms, any one-dimensional closed subgroup $L\subset \HB$ is given by
\begin{enumerate}
\item $L=\mathbb{Z}^k\times\mathbb{R}$ for $k=0, 1, 2$  or 
\item $L=\mathbb{R}\mathbf{e}_1\times\mathbb{Z} p$ for $p=0, 1$.
\end{enumerate}
\end{proposition}
We remark that, up to isomorphisms, $\mathbb{Z}^{k}\times \mathbb{R}$ are the normal subgroups of $\HB$ for $k=0,1,2$. We will exclude the consideration of the normal subgroups of $\HB$. Doing so is intentional, as the corresponding homogeneous spaces become Lie groups, and the study of LCSs on such spaces has already been thoroughly investigated in a series of papers. For a detailed exposition, refer to \cite{Ayala, VA1, VAP, VA2, VA3}.

In our earlier paper \cite{AEO}, we have classified both normal and
non-normal closed subgroups of dimension one of $\mathbb{H}$ and the main
focus was given to the structure of homogeneous spaces through non-normal
ones since the other case was already treated in the literature in a series
of recent papers. These non-normal subgroups of $\mathbb{H}$ are (i) $%
L=\mathbb{R}\mathbf{e}_{1}\times \mathbb{Z}$ and (ii) $L=\mathbb{R}\mathbf{e}%
_{1}\times \{0\}$ and the structure of all LCSs on the corresponding homogeneous spaces $
L\backslash \mathbb{H}\simeq
\mathbb{R}\times \mathbb{T}$ and $L\backslash \mathbb{H}\simeq \mathbb{R}\times \mathbb{R}$ have been
exposed. Note that any induced control system $\Sigma$ on $ L\backslash \mathbb{H} $ is equivalent to one of the systems in the
diagram (\ref{diagram1}) below. We remember that
controllability of such a system on the homogeneous space $\mathbb{R}\times \mathbb{T}$ has already been studied in \cite{AEO} and the
purpose of this article is to consider the controllability of the induced
control systems on the other homogeneous space $\mathbb{R}\times \mathbb{R}$%
. The latter case that we are going to deal with here has much more complex
structures and appears to be quite complicated.
\begin{equation}\label{diagram1}
\begin{tikzcd}[row sep=large,column sep=huge]
 & ~\Sigma~ \arrow[dl,"L\backslash \HB\simeq\R\times\R"'] \arrow[dr,"L
\backslash  \HB\simeq \mathbb{R}\times \mathbb{T}"] &  \\
 \Sigma_{1,0} & & \Sigma_{1,1}
\end{tikzcd}
\end{equation}

To recognize and distinguish better the system(s) that are already considered
before (i.e., $\Sigma_{1,1}$) and the one(s) we take into account now (i.e., $\Sigma_{1,0}$) we find it convenient to quote
from our earlier paper, the following proposition without proof:

\begin{proposition}\label{MainPro}\cite[Proposition 3]{AEO}
    	Let $\XC=(\eta, A)$ be a linear vector field on $\HB$ with associated flow $\{\varphi_t\}_{t\in\mathbb{R}}$. It holds that
\begin{enumerate}
    \item $\mathbb{R} \mathbf{e}_{1}\times\{0\}$ is $\varphi_t$-invariant if and only if 
    		$$A \mathbf{e}_{1}=\lambda \mathbf{e}_{1}, \;\;A\mathbf{e}_{2}=\beta \mathbf{e}_{2}+\alpha \mathbf{e}_{1}\;\;\mbox{ and }\;\;\eta\in \mathbb{R} \mathbf{e}_{2}, \;\mbox{ with }\;\alpha=0 \;\mbox{ if }\;\eta\neq 0;$$
     \item $\mathbb{R} \mathbf{e}_{1}\times\mathbb{Z}$ is $\varphi_t$-invariant if and only if 
    		$$A \mathbf{e}_{1}=\lambda \mathbf{e}_{1}, \;\;A\mathbf{e}_{2}=-\lambda \mathbf{e}_{2}+\alpha \mathbf{e}_{1}\;\;\mbox{ and }\;\;\eta\in\mathbb{R}\mathbf{e}_{2}, \;\mbox{ with }\;\alpha=0 \;\mbox{ if }\;\eta\neq 0.$$ 
    		
    		\end{enumerate} 
      \end{proposition}	
Assume that the subgroup $L=\mathbb{R}\mathbf{e}_1\times \{0\}$ is invariant under the flow of $\mathcal{X}=(\eta, A)$. By Proposition \ref{MainPro}, it follows that
			 $$\eta=\left(\begin{array}{c}
			      0  \\
			      \gamma 
			 \end{array}\right)\;\;\;\mbox{ and }\;\;\;A=\left(\begin{array}{cc}
			 	\lambda & \alpha\\ 0 & \beta
			 \end{array}\right), \;\;\mbox{ with }\alpha=0\;\mbox{ if }\;\gamma\neq 0.$$
		     Then in coordinates, we have
		     $$\XC((x, y), z)=((\lambda x+\alpha y, \beta y), \gamma y+(\lambda+\beta)z).$$
Hence, the general expression of a vector field on the homogeneous space  $L\backslash \HB$ induced by a linear vector field on $\HB$ is given by 
$$\widehat{\XC}_{1, 0}(s, t)=\left(\beta s, (\lambda+\beta)t+\alpha s^2+\gamma s\right), \;\;\mbox{ with }\;\;\alpha=0\;\mbox{ if }\;\;\gamma\neq 0.$$
Similar calculations lead us to obtain a typical induced invariant vector field on the homogeneous space  $L\backslash \HB$ as
$$\widehat{B}_{1, 0}(s, t)=\left(b, c+as\right), \;\;(s, t)\in\mathbb{R}\times\mathbb{R}$$ so that a one-input linear control system $\Sigma _{1,0}$ on $L \backslash \mathbb{H}\simeq
\mathbb{R}\times \mathbb{R}$ has the form (see \cite[Proposition 5]{AEO})
	\begin{flalign*}
   \label{oneinput1D}
		&&\left(\Sigma_{1, 0}\right):\left\{
		\begin{array}{l}
	 		\dot{s}=\beta s+\omega b\\
	 		\dot{t}=(\lambda+\beta)t+\frac{1}{2}\alpha s^2+\gamma s+\omega(c+as)
		\end{array}\right.,  &&
	\end{flalign*}
 where $a, b, c, \alpha, \beta, \gamma\in\R$ with $\alpha=0$ if $\gamma\neq 0$.

The following proposition characterizes the Lie algebra rank condition of the system $\Sigma_{1, 0}$ which is indispensable for the controllability issue.

\begin{proposition}
  \label{oneinput1D}
	 	The one-input LCS $\Sigma_{1, 0}$ on $L\backslash \mathbb{H}\simeq\R\times\R$, 
 satisfies the LARC if and only if 
	 	$$b\cdot \left((b\alpha+a(\lambda-\beta))^2+(b\gamma+c\lambda)^2\right)\neq 0.$$
	 \end{proposition}
\begin{proof}
 Let us show that $\mathrm{span}_{\mathcal{L}A}\{\widehat{\XC}_{1, 0}, \widehat{B}_{1, 0}\}(s,t)= \R^2 $ for all $(s,t)\in \R \times \R$. Firstly, looking at the Lie bracket of $\widehat{\XC}_{1, 0}$ and $\widehat{B}_{1, 0}$ we have that
 \begin{align*}
 [\widehat{\XC}_{1, 0}, \widehat{B}_{1, 0}]=&\bigg(-b\beta ,(a\beta s-b(\alpha s +\gamma)-(\lambda +\beta)(c+as)) \bigg) \\
 =& -\beta \widehat{B}_{1, 0} - \bigg\{ \underbrace{(0,b\alpha+\lambda a - a \beta)}_{:=Z_1}s + \underbrace{(0,b\gamma + \lambda c)}_{:=Z_2} \bigg\} \\
 =& -\beta \widehat{B}_{1, 0} -  \{sZ_1+Z_2\} = -\beta \widehat{B}_{1, 0} - Z.
 \end{align*}
Then let's consider the other brackets, respectively:
\begin{align*}
 [sZ_1 , \widehat{\XC}_{1, 0}] &= \lambda s Z_1  && [sZ_1 , \widehat{B}_{1, 0}] = -b Z_1, \\
 [Z_2 , \widehat{B}_{1, 0}] &=0  && [Z_2 , \widehat{\XC}_{1, 0}] = (\lambda+\beta) Z_2, \\
 [[\widehat{\XC}_{1, 0}, \widehat{B}_{1, 0}],\widehat{\XC}_{1, 0}] &=-\beta^2 \widehat{B}_{1, 0}-(\lambda+\beta)Z - \beta Z_2 && [[\widehat{\XC}_{1, 0}, \widehat{B}_{1, 0}],\widehat{B}_{1, 0}]=bZ_1.
\end{align*}
If it is continued in this process, we see that all brackets just depend on the vector fields $\widehat{\XC}_{1, 0},\widehat{B}_{1, 0}, Z_1, Z_2$ and $Z$. Finally, one can obtain that LARC satisfied if and only if $bZ \neq 0$. Thus, the proof is complete.
\end{proof} 

In what follows, we define singular LCSs.

\begin{definition}
\label{singular}
We say that $\Sigma _{1,0}$ on $L \backslash \mathbb{H}\simeq
\mathbb{R}\times \mathbb{R}$ is singular if the associated vector field 
$$\widehat{\XC}_{1, 0}(s, t)=\left(\beta s, (\lambda+\beta)t+\alpha s^2+\gamma s\right),$$
satisfies $\beta(\lambda+\beta)=0.$
\end{definition}

The precise description of the dynamics of the LCS $\Sigma_{1, 0}$ is quite difficult. Because of that our study in the present paper will be focused on singular systems.

\begin{remark}
  The previous definition is related to the set of singularities of the induced vector field $\widehat{\XC}_{1, 0}$. In fact, by a simple calculation, one sees that $\beta(\lambda+\beta)\neq 0$ if and only if $(0, 0)$ is the only singularity of $\widehat{\XC}_{1, 0}$.  
\end{remark}

 \section{Controllability and control sets of $\Sigma_{1, 0}$}\label{controlsec}
Controllability and the control sets of LCSs on the homogeneous spaces of $%
\mathbb{H}$ satisfying the LARC are thoroughly classified in this section. By the previous sections, any LCS on a homogeneous space $L\backslash \mathbb{H}
$ of $\mathbb{H}$ is equivalent to one of the control-affine systems $\Sigma
_{1,0}$ and $\Sigma _{1,1}$. In what follows, we will make a detailed
analysis of the system $\Sigma _{1,0}$ since the other one was already
studied in our earlier work.

It follows from the singularity imposed on the associated vector field that we should treat the possibilities
$\beta(\lambda+\beta)=0$. In what follows, we do a
detailed analysis of the possible control sets of the one-input LCS on %
\mbox{$(\mathbb{R}\mathbf{e}%
_{1}\times \{0\})\backslash\HB$.} We divide such an analysis
in two main cases depending on $\alpha$ together with their subcases (whenever they exist) depending on the eigenvalues of $A$. See the figure below.

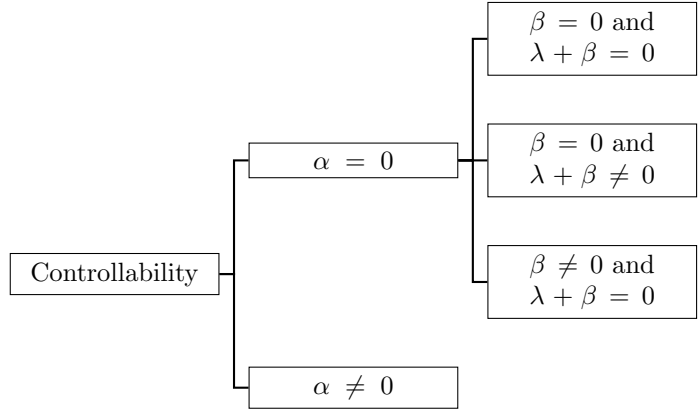
\begin{figure}[H]
\centering
\begin{tikzpicture}[grow'=right,level distance=1.25in,sibling distance=.25in]
\tikzset{edge from parent/.style= 
            {thick, draw, edge from parent fork right},
         every tree node/.style=
            {draw,minimum width=1in,text width=1in,align=center}}
\Tree 
    [. Controllability 
        [.{$\alpha=0$}
                [.{$\beta=0$ and $\lambda+\beta =0$ } ]
            [.{$\beta =0$ and $\lambda+\beta \neq 0$ } ]
            [.{$\beta \neq 0$ and $\lambda+\beta = 0$ } ]            
        ]
        [.$\alpha \neq 0$ ]
    ]
\end{tikzpicture}
\caption{Description of all possible cases to be analyzed on}\label{casesss}
\end{figure}
 
 \subsection{Controllability in the case $\alpha=0$.}
In this section, we will be assuming that $\alpha=0$, and hence our system does not have a quadratic term. The analysis for this case consists of four parts and is done by analyzing the possibilities for the values of $\beta$ and $\lambda+\beta$ as follows:
 
\subsubsection{The subcase $\beta=\lambda+\beta=0$:}

  In this case, the system is given by 
  
  \begin{flalign*}
	 	&&\left(\Sigma_{1, 0}\right):\left\{
	 	\begin{array}{l}
	 		\dot{s}=\omega b\\
	 		\dot{t}=(\gamma +a\omega )s+ c\omega
	 	\end{array}\right. &&
	 \end{flalign*}	
where $\omega\in\Omega$. By Proposition \ref{oneinput1D}, it satisfies the LARC if and only if $b\gamma\neq 0$.
 
 \begin{proposition}
     The system $\Sigma_{1, 0}$ is controllable.
 \end{proposition}
\begin{proof}
For $\omega\neq 0$ and $\mathbf{v}_0=(s_0, t_0)$, the solutions $\varphi(\tau, \mathbf{v}_0, \omega)$ of $\Sigma_{1, 0}$ are given, component-wise, as
\begin{equation}	
 \begin{aligned}
		\label{sola}
	\varphi_1(\tau, \mathbf{v}_0, \omega)&=s_0+\omega b\tau \\
 \varphi_2(\tau, \mathbf{v}_0, \omega)&=\frac{b\omega(a\omega+\gamma)}{2}\tau^2+\left(s_0(a\omega+\gamma)+c\omega\right)\tau+t_0.
\end{aligned}
\end{equation}
	A simple calculation shows that such a solution coincides with the parabola 
	$$\Gamma_{\mathbf{v}_0, \omega}:=\left\{\left(s, \frac{(a\omega+\gamma)}{2b\omega}(s-s_0)^2+\frac{s_0(a\omega+\gamma)+c\omega}{b\omega}(s-s_0)+t_0\right)~:~ s\in\R\right\},$$	
	with concavity determined by the sign of $b\omega$. Let us analyze the case where $\gamma, b\in\R^+$, since other choices are treated similarly. For this choice, there exists $\varepsilon>0$ such that $a\omega+\gamma>0$ for $\omega\in(-\varepsilon, \varepsilon)$.  
	
	Let
	$$\mathbf{v}_0\neq \mathbf{v}_1\in\R^2\;\;\;\mbox{ and } \;\;\;-\varepsilon<\omega_0<0<\omega_1<\varepsilon.$$ 
	A trajectory connecting $\mathbf{v}_0$ to $\mathbf{v}_1$ can be constructed in the two steps as follows:
	
	\begin{itemize}
		\item[Step 1.] $\mathbf{v}_0$ belongs to the interior of the region determined by the parabola $\Gamma_{\mathbf{v}_1, \omega_1}$.

		Since $b, \gamma\in\R^+$ we have that $\tau\rightarrow+\infty$ implies that 
		$$\varphi_1(\tau, \mathbf{v}_0, \omega_0)\rightarrow-\infty\hspace{.5cm}\mbox{ and }\hspace{.5cm}\varphi_2(\tau, \mathbf{v}_0, \omega_0)\rightarrow-\infty.$$
		As a consequence, there exists $\tau_0>0$ such that $\tilde{\mathbf{v}}_1:=\varphi(\tau_0, \mathbf{v}_0, \omega_0)$ belongs to $\Gamma_{\mathbf{v}_1, \omega_1}$. Write $\tilde{\mathbf{v}}_1=(\tilde{s}_1, \tilde{t}_1)$ and consider the following cases:

  \begin{enumerate}
      \item[(i)]  $\tilde{s}_1\leq s_1$ :
		
		Since the solution starting at $\tilde{\mathbf{v}}_1$ associated to the control $\omega_1$ lies on the parabola $\Gamma_{\mathbf{v}_1, \omega_1}$ and  
	    $$\tau\rightarrow+\infty\hspace{.5cm}\implies\hspace{.5cm} \varphi_1(\tau, \tilde{\mathbf{v}}_1, \omega_1)\rightarrow+\infty,$$
	    there exists $\tau_1>0$ such that $\varphi(\tau_1, \tilde{\mathbf{v}}_1, \omega_1)=\mathbf{v}_1$. By concatenation, we get a trajectory from $\mathbf{v}_0$ to $\mathbf{v}_1$ (Figure \ref{dentro1}).
		
		\item[(ii)]  $\tilde{s}_1> s_1$:
		
		Since the parabolas $\Gamma_{\mathbf{v}_1, \omega_0}$ and $\Gamma_{\mathbf{v}_0, \omega_0}$ coincides with solutions of $\Sigma_{1, 0}$ for the constant control $\omega_0$, they are parallels. Therefore, the assumption $\tilde{s}_1>s_1$ implies that $\mathbf{v}_0$ lies in the interior of the region determined $\Gamma_{\mathbf{v}_1, \omega_0}$. On the other hand, $\tau\rightarrow+\infty$ implies 
		$$\varphi_1(\tau, \mathbf{v}_0, \omega_1)\rightarrow+\infty\hspace{.5cm}\mbox{ and }\hspace{.5cm}\varphi_2(\tau, \mathbf{v}_0, \omega_1)\rightarrow+\infty,$$
		and consequently, there exists $\tau_0>0$ such that $\tilde{\mathbf{v}}_0:=\varphi(\tau_0, \mathbf{v}_0, \omega_1)$ belongs to the parabola determined by $\mathbf{v}_1$ and $\omega_0$. By writing $\tilde{\mathbf{v}}_0=(\tilde{s}_0, \tilde{t}_0)$ it holds that $\tilde{s}_0> s_1$ and hence, $\varphi(\tau_1, \tilde{\mathbf{v}}_0, \omega_0)=\mathbf{v}_1$ for some $\tau_1>0$. By concatenation we obtain a trajectory from $\mathbf{v}_0$ to $\mathbf{v}_1$ (Figure \ref{dentro2}).
	\end{enumerate}

 \begin{figure}[H]
	\centering
	\begin{subfigure}{.38\textwidth}
		\centering
		\includegraphics[width=1\linewidth]{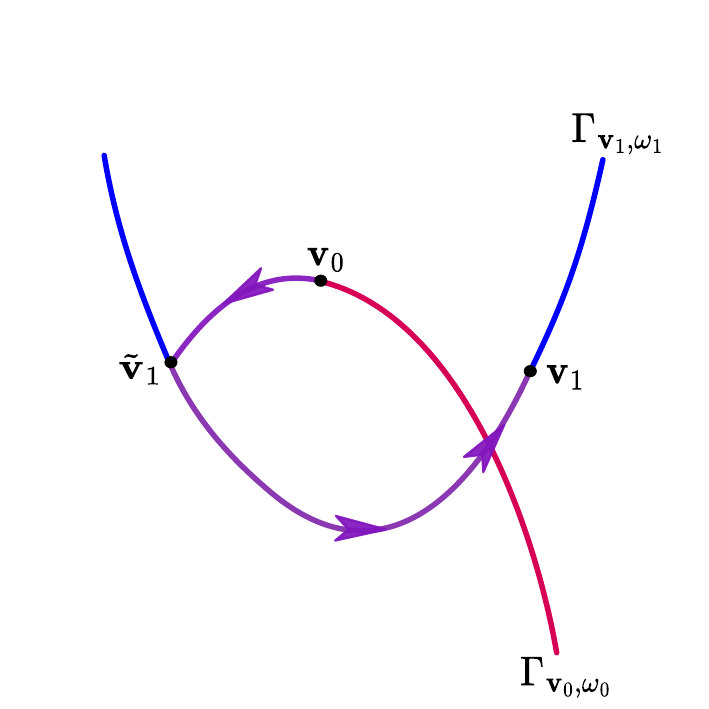}
		\caption{Trajectory starting inside $\Gamma_{\mathbf{v}_1, \omega_1}$ with $\tilde{s}_1\leq s_1$}
		\label{dentro1}
	\end{subfigure}%
 \hspace{1cm}
	\begin{subfigure}{.38\textwidth}
		\centering
		\includegraphics[width=1\linewidth]{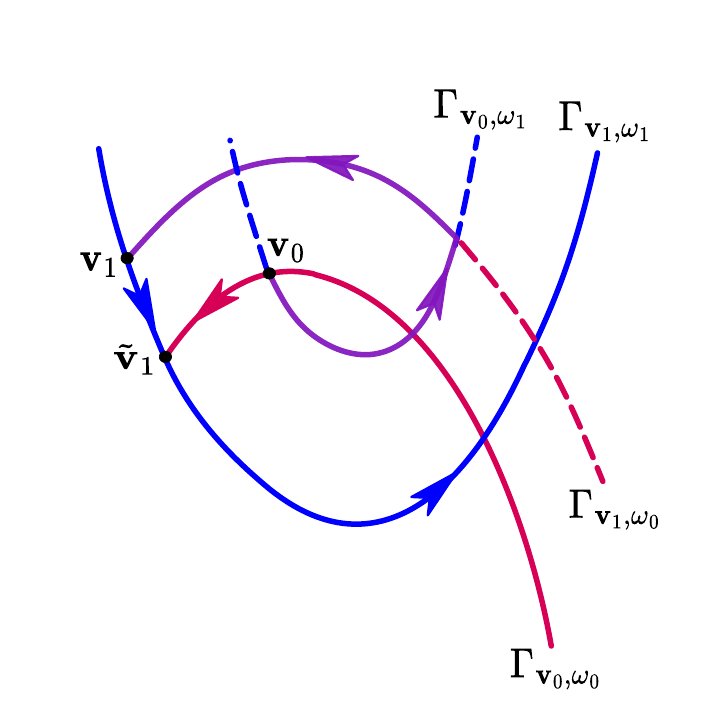}
		\caption{\mbox{Trajectories starting inside $\Gamma_{\mathbf{v}_1, \omega_1}$ with $\tilde{s}_1> s_1$}}
  \label{dentro2}
	\end{subfigure}
\end{figure}
		
\item[Step 2.] $\mathbf{v}_0$ belongs to the exterior of the region determined $\Gamma_{\mathbf{v}_1, \omega_1}$.

  Let us show, in this case, the existence of a trajectory connecting $\mathbf{v}_0$ to a point in the interior $\Gamma_{\mathbf{v}_1, \omega_1}$, which by item (a) implies the result.

Since the interior of $\Gamma_{\mathbf{v}_1, \omega_1}$ is the set
  $$\left\{(s, t)\in\R^2~:~  t>\frac{(a\omega_1+\gamma)}{2b\omega_1}(s-s_1)^2+\frac{s_1(a\omega_1+\gamma)+c\omega_1}{b\omega_1}(s-s_1)+t_1\right\},$$
  we construct a trajectory from $\mathbf{v}_0$ to a point in this region as follows:

  \begin{enumerate}
      \item[(i)] Since $$\tau\rightarrow+\infty\hspace{.5cm}\implies\hspace{.5cm}\varphi_1(\tau, \mathbf{v}_0, \omega_1)\rightarrow +\infty,$$
  there exists $\tau_0>0$ such that the point 
  $\varphi(\tau_0, \mathbf{v}_0, \omega_1):=\tilde{\mathbf{v}}_0=(\tilde{s}_0, \tilde{t}_0)$ satisfies $\tilde{s}_0>0$;

  \item [(ii)]  Now, with control $\omega= 0$ we get that
  $$\varphi(\tau, \tilde{\mathbf{v}}_0, 0)=(\tilde{s}_0, \tilde{s}_0\gamma\tau+\tilde{t}_0),$$
  which for $\tau>0$ large enough satisfies 
$$\tilde{s}_0\gamma\tau+\tilde{t}_0>\frac{(a\omega_1+\gamma)}{2b\omega_1}(\tilde{s}_0-s_1)^2+\frac{s_1(a\omega_1+\gamma)+c\omega_1}{b\omega_1}(\tilde{s}_0-s_1)+t_1,$$
assuring the existence of $\tau_1>0$ such that $\phi(\tau_1, \tilde{\mathbf{v}}_0, 0)$ belongs to the interior of $\Gamma_{\mathbf{v}_1, \omega_1}$ as stated. 
  \end{enumerate}

  
  \end{itemize}
  
\end{proof}

\subsubsection{The subcase $\beta=0, \lambda+\beta\neq0$:}

In this case, $b\lambda\neq 0$ and the diffeomorphism 
  $$f:\mathbb{R}^2\rightarrow \R^2,\hspace{1cm}f(s, t)=\left(\frac{s}{b}, t+\frac{\gamma}{\lambda}s\right),$$
  conjugates our initial system to the system
   \begin{flalign*}
	 	&&\left(\Sigma_{1, 0}\right):\left\{
	 	\begin{array}{l}
	 		\dot{s}=\omega\\ \dot{t}=\lambda t+\omega(c+as)\end{array}\right.  &&
	 \end{flalign*}	
where the LARC is equivalent to  $a^2+c^2\neq 0 $. Let us assume w.l.o.g. that $\lambda<0$ and $a\geq 0$,  since the other cases are analogous.

\bigskip

For each $\omega\in \Omega$,  let us define the function 
$$F_{\omega}:\R^2\rightarrow \mathbb{R}, \hspace{1cm}F_{\omega}(s, t)=\lambda^2t+\omega(\lambda(c+as)+a \omega).$$

By a straightforward calculation, one shows that, for $\omega\in\Omega$ and $\mathbf{v}=(s, t)$, the first coordinate of the solution of $\Sigma_{1, 0}$ is given by $\varphi_1(\tau, \mathbf{v}, \omega)=s+\tau \omega$, and the second one is determined by the relation
 \begin{equation}
     \label{retas}
     F_{\omega}\left(\varphi(\tau, \mathbf{v}, \omega)\right)=\rme^{\lambda \tau}F_{\omega}(\mathbf{v}).
 \end{equation}
 
Moreover, for any $\omega_0, \omega_1, \omega_2\in\Omega$, it holds that
\begin{equation}
    \label{twocontrol}
    F_{\omega_1}\left(\varphi(\tau, \mathbf{v}, \omega_0)\right)-F_{\omega_2}\left(\varphi(\tau, \mathbf{v}, \omega_0)\right)=a\omega_0(\omega_1-\omega_2)\lambda\tau+F_{\omega_1}(\mathbf{v})-F_{\omega_2}(\mathbf{v}), \;\;\;\forall\tau \in\mathbb{R}, \mathbf{v}\in\mathbb{R}^2.
\end{equation}

Note that $F^{-1}_{\omega}(0)$ is a line on the $\R^2$ whose inclination, with relation to the $s$-axis, and intersection with the $t$-axis are given, respectively, by 
$$-\frac{a \omega}{\lambda}\hspace{.5cm}\mbox{ and }\hspace{.5cm}-\frac{\omega}{\lambda^2}(c\lambda +a\omega).$$ 

By relation (\ref{retas}) and the assumption that $\lambda<0$, the line $F^{-1}_{\omega}(0)$ is an asymptote of the curve $\tau\mapsto\varphi(\tau, \mathbf{v}, u)$ as $\tau\rightarrow+\infty$ (see Figure \ref{fig2}). 
\begin{figure}[H]
	\centering
	\includegraphics[scale=.31]{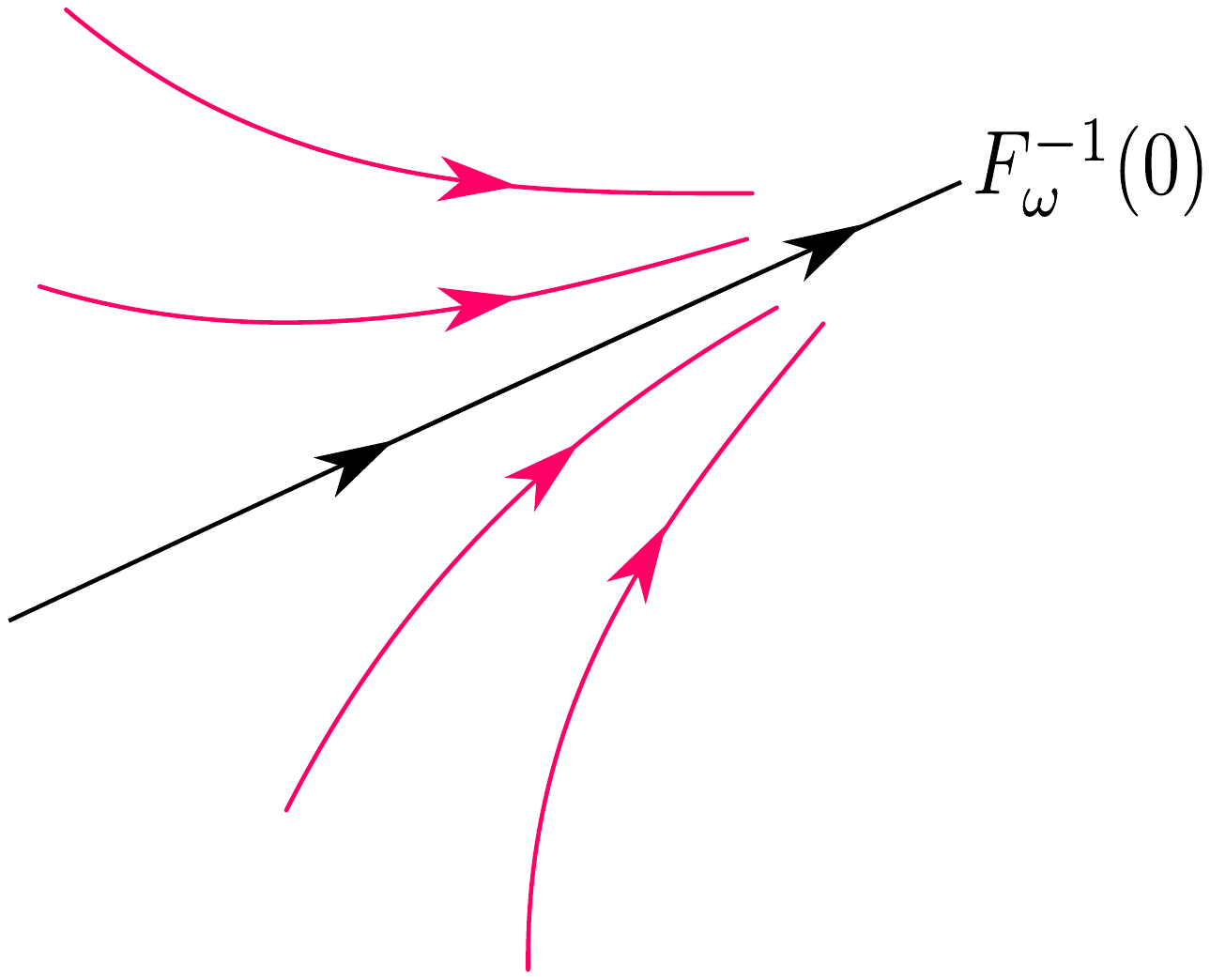}
	\caption{The lines $F^{-1}_{\omega}(0)$ are asymptotes of the curve $\tau\mapsto\varphi(\tau, \mathbf{v}, u)$ as $\tau\rightarrow+\infty$}
	\label{fig2}
\end{figure}

Assume that $a\neq 0$, and let us define some regions with invariant properties. The first one is the region $\CC^-$ given by
$$\CC^-:=\left\{ \mathbf{v}\in\mathbb{R}^2~:~ F_{\omega^-}(\mathbf{v})<0 \;\mbox{ and }\;F_{\omega^+}(\mathbf{v})<0\right\} .$$
For the second one, let us consider the point $\mathbf{v}_a=(-c/a, 0)$ and define 
$$\CC^+=\Bigl\{(s, t)\in\R^2~:~\exists\tau\geq 0 \mbox{ and }\omega\in\{\omega^-, \omega^+\} \mbox{ with } s=\varphi_1(\tau, \mathbf{v}_a, \omega)\mbox{ and }t>\varphi_2(\tau, \mathbf{v}_a, \omega)\Bigr\}.$$
Geometrically, $\CC^-$ is the set of points under the lines $F_{\omega^-}^{-1}(0)$ and $F_{\omega^+}^{-1}(0)$ and $\CC^+$  the set of points over the curves 
$$\tau\in(0, +\infty)\mapsto \varphi(\tau, \mathbf{v}_a, \omega^-)\;\;\;\mbox{ and }\;\;\; \tau\in(0, +\infty)\mapsto \varphi(\tau, \mathbf{v}_a, \omega^+),$$
as depicted in Figure \ref{fig1} below. 
\begin{figure}[H]
	\centering
	\includegraphics[scale=.38]{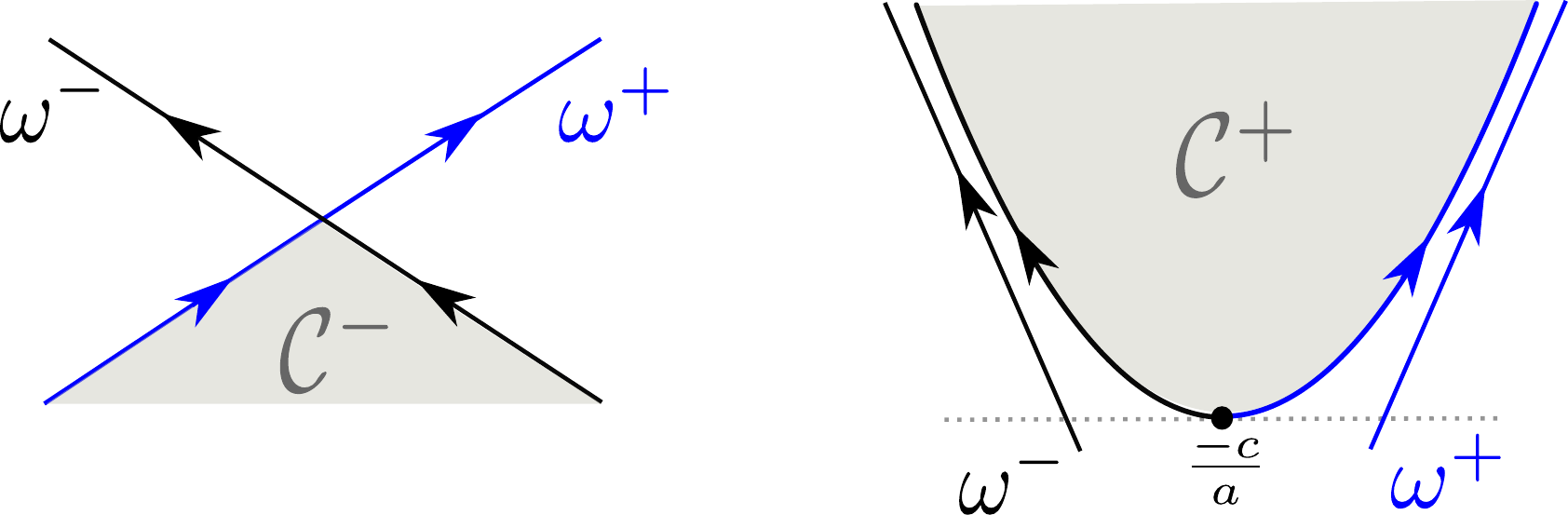}
	\caption{A geometrical description of the regions $\CC^-$ and $\CC^+$}
	\label{fig1}
\end{figure}

The solutions starting at the point $\mathbf{v}_a$ for constant control can be explicitly calculated as follows: Since $F_{\omega}(\mathbf{v}_a)=a\omega^2$, relation (\ref{retas}) gives us that
$$\lambda^2\varphi_2(\tau, \mathbf{v}_a, \omega)=-\omega(\lambda(a\underbrace{(-c/a+\tau\omega)}_{\varphi_1(\tau, \mathbf{v}_a, \omega)}+c)+a\omega)+a\omega^2\rme^{\lambda\tau}=a\omega^2(\rme^{\lambda\tau}-\lambda\tau-1),$$
implying that 
\begin{equation}
    \label{solution}
    \varphi(\tau, \mathbf{v}_a, \omega)=\left(-\frac{c}{a}+\tau\omega~,~ \frac{a\omega^2}{\lambda^2}(\rme^{\lambda\tau}-\lambda\tau-1)\right).
\end{equation}

The next lemma assures that $\CC^-$ and $\CC^+$ are invariant in negative time and also proves controllability in some regions determined by the lines $F_{\omega}^{-1}(0)$.

\begin{lemma}
\label{lema}
Assume $a>0$. It holds:

\begin{itemize}

\item[(1)] The regions $\CC^-$ and $\CC^+$ are invariant in negative time.

\item[(2)] For any $\omega_1<0<\omega_2$, the region
$$\mathcal{C}(\omega_1, \omega_2)=\{\mathbf{v}\in\mathbb{R}^2~:~ F_{\omega_1}(\mathbf{v})\cdot F_{\omega_2}(\mathbf{v})<0\},$$
is controllable.

\end{itemize}

\end{lemma}

\begin{proof} 
(1)  Note that, for fixed $\mathbf{v}\in\mathbb{R}^2$, the map $\omega\mapsto F_{\omega}(\mathbf{v})$ is a polynomial with a maximum degree equal to two. Since we are assuming $a\geq 0$ we have that 
$$ F_{\omega^-}(\mathbf{v})<0\;\;\mbox{ and }\;\; F_{\omega^+}(\mathbf{v})<0\;\;\iff\;\; F_{\omega}(\mathbf{v})<0, \;\;\forall\omega\in\Omega.$$

Let $\omega, \widehat{\omega}\in\Omega$ and define the function
$$\psi:\R\rightarrow\R, \hspace{1cm} \psi(\tau)= F_{\omega}(\varphi(\tau, \mathbf{v}, \widehat{\omega})).$$
In order to show the invariance of $\mathcal{C}^-$ in negative time, it is enough to show that $\psi(\tau)<0$ if  $\tau<0$. However, relations (\ref{retas}) and (\ref{twocontrol}), allow us to rewrite $\psi$ as
$$\psi(\tau)= F_{\omega}(\varphi(\tau, \mathbf{v}, \widehat{\omega}))-F_{\omega_2}(\varphi(\tau, \mathbf{v}, \widehat{\omega}))+F_{\omega_2}(\varphi(\tau, \mathbf{v}, \widehat{\omega}))$$
$$=a\widehat{\omega}(\omega-\widehat{\omega})\lambda \tau+F_{\omega}(\mathbf{v})+(\rme^{\lambda\tau}-1)F_{\widehat{\omega}}(\mathbf{v}).$$

Since,
$$\widehat{\omega}(\omega-\widehat{\omega})\leq 0\hspace{.5cm}\implies\hspace{.5cm} \psi(\tau)<0, \hspace{.5cm}\mbox{ for }\hspace{.5cm}\tau<0,$$
we can assume, w.l.o.g., that $\widehat{\omega}(\omega-\widehat{\omega})>0$. On the other hand, if $a\widehat{\omega}(\omega-\widehat{\omega})+F_{\widehat{\omega}}(\mathbf{v})>0,$
we get that
$$a\widehat{\omega}(\omega-\widehat{\omega})+F_{\widehat{\omega}}(\mathbf{v})=\underbrace{\lambda^2t}_{F_0(\mathbf{v})<0}+\;\widehat{\omega}(\lambda(as+c)+a\omega)\implies \widehat{\omega}(\lambda(as+c)+a\omega)>0,$$
and hence,
$$0>F_{\omega}(\mathbf{v})=a\widehat{\omega}(\omega-\widehat{\omega})+F_{\widehat{\omega}}(\mathbf{v})+(\omega-\widehat{\omega})(\lambda(as+c)+a\omega)$$
$$=a\widehat{\omega}(\omega-\widehat{\omega})+F_{\widehat{\omega}}(\mathbf{v})+\frac{1}{\widehat{\omega}^2}\Bigl[\widehat{\omega}(\omega-\widehat{\omega})\Bigr]\Bigl[\widehat{\omega}(\lambda(as+c)+a\omega)\Bigr]>0,$$
which is absurd. Therefore, $a\widehat{\omega}(\omega-\widehat{\omega})+F_{\widehat{\omega}}(\mathbf{v})\leq 0$ and 
$$\psi(\tau)=a\widehat{\omega}(\omega-\widehat{\omega})\lambda \tau+F_{\omega}(\mathbf{v})+(\rme^{\lambda\tau}-1)F_{\widehat{\omega}}(\mathbf{v})$$
$$
=\left(a\widehat{\omega}(\omega-\widehat{\omega})+F_{\widehat{\omega}}(\mathbf{v})\right)\lambda\tau+F_{\omega}(\mathbf{v})+(\rme^{\lambda\tau}-\lambda\tau-1)F_{\widehat{\omega}}(\mathbf{v})<0, \;\;\;\mbox{ if }\;\;\;\tau<0,$$
showing the invariance in negative-time of $\CC^-$.

Let us now show the invariance of $\mathcal{C}^+$ in negative time. Note first that, by the very definition, 
$$(s, t)\in\mathcal{C}^+\;\;\;\implies\;\;\; \{(s, t+\rho), \rho \geq 0\}\subset\CC^+,$$
and hence,
$$\varphi(\tau, (s, t), 0)=(s, \rme^{\lambda \tau}t)\in\{(s, t+\rho), \rho\geq 0\}\;\;\;\mbox{ when }\;\;\;\tau<0,$$
showing that $\varphi(\tau, \CC^+, 0)\subset\CC^+$. On the other hand, the expression of the first coordinate of the solutions of the system together with equation (\ref{solution}) imply that, for any $\mathbf{v}\in\CC^+$, the curve $\tau\in\R\mapsto\varphi(\tau, \mathbf{v}, \omega)\in\R^2$ intersects $-c/a\times(0, +\infty)$ at exactly one point. Since solutions of ODEs are either parallel or coincident, in order to show the invariance of $\CC^+$, it is enough to show that the curves $\tau\in(-\infty, 0)\mapsto\varphi(\tau, \mathbf{v}_a, \omega)$ for $\omega\in\Omega$ with $\omega\neq 0$ do not leave $\mathcal{C}^+$.

Let then $\omega\in\Omega$ with $\omega\neq 0$ and $\tau_1<0$. There exists $\tau_2>0$ such that $\tau_1\omega=\tau_2\omega^i$, where
$$\omega^i=\left\{\begin{array}{ll}
    \omega^+, & \omega<0 \\
     \omega^- ,& \omega>0
\end{array}\right.,$$
Therefore,
\begin{align*}
 \varphi_2(\tau_1, \mathbf{v}_a, \omega)&=\frac{a\omega^2}{\lambda^2}(\rme^{\lambda\tau_1}-\lambda\tau_1-1)=\frac{a(\tau_1\omega)^2}{\lambda^2}\frac{1}{\tau_1^2}(\rme^{\lambda\tau_1}-\lambda\tau_1-1) \\
 &>\frac{a(\tau_2\omega^i)^2}{\lambda^2}\frac{1}{\tau_2^2}(\rme^{\lambda\tau_2}-\lambda\tau_2-1)=\varphi_2(\tau_2, \mathbf{v}_a, \omega^i)
\end{align*}
where, for the inequality, we used that the function
$$f:\R\rightarrow\R, \hspace{1cm} f(\tau)=\left\{\begin{array}{ll}
   \frac{1}{\tau^2}(\rme^{\lambda\tau}-\lambda\tau-1),  &  \tau\neq 0\\
    \frac{\lambda^2}{2}, & \tau=0
    \end{array}\right.,$$
is strictly decreasing when $\lambda<0$. As a consequence, $\varphi(\tau, \CC^+, \omega)\subset\CC^+$ for any $\tau<0$, showing the result.

\hspace{1.1cm} (2) In order to show this item, it is enough to show that for any $\mathbf{v}_{1}, \mathbf{v}_{2}\in\mathbb{R}^2$ satisfying 
$$F_{\omega_1}(\mathbf{v}_{1})>0, \hspace{.5cm} F_{\omega_2}(\mathbf{v}_{1})<0\hspace{.5cm}\mbox{ and }\hspace{.5cm} F_{\omega_1}(\mathbf{v}_{2})<0, \hspace{.5cm}F_{\omega_2}(\mathbf{v}_{2})>0,$$
there exists a closed orbit passing by $\mathbf{v}_{1}$ and $\mathbf{v}_{2}$.

Define the curve
$$\gamma_1:\mathbb{R}\rightarrow\mathbb{R},\hspace{1cm}\gamma_1(\tau)=F_{\omega_2}(\varphi(\tau, \mathbf{v}_{1}, \omega_1)).$$
As in the previous item, we have that 
$$\gamma_1(\tau)=a\omega_1(\omega_2-\omega_1)\lambda\tau+F_{\omega_2}(\mathbf{v}_{1})+(\rme^{\lambda\tau}-1)F_{\omega_1}(\mathbf{v}_{1}).$$
Derivation gives us that 
$$\gamma_1'(\tau)=\lambda[a\omega_1(\omega_2-\omega_1)+\rme^{\lambda\tau}F_{\omega_1}(\mathbf{v}_{1})]\hspace{.5cm}\mbox{ and }\hspace{.5cm}\gamma''(\tau)=\lambda^2\rme^{\lambda\tau}F_{\omega_1}(\mathbf{v}_{1}),$$
showing that $\gamma_1'$ is strictly increasing. Since, $F_{\omega_1}(\mathbf{v}_{1})>0$ and $\omega_1(\omega_2-\omega_1)<0$ we conclude that $\gamma'$ has exactly one zero. As a consequence, 
$$\lim_{\lambda\tau\rightarrow\pm\infty}\gamma(\tau)=+\infty\hspace{.5cm}\mbox{ and }\hspace{.5cm}\gamma_1(0)=F_{\omega_2}(\mathbf{v}_{1})<0,$$
imply that the curve $\tau\in\mathbb{R}\mapsto\varphi(\tau, \mathbf{v}_{1}, \omega_1)$ crosses the line $F_{\omega_2}^{-1}(0)$ exactly two times, one in positive time and one in negative time  (see Figure \ref{fig3}). A similar analysis of the curve $$\gamma_2:\mathbb{R}\rightarrow\mathbb{R},\hspace{1cm}\gamma_2(\tau)=F_{\omega_1}(\varphi(\tau, \mathbf{v}_{2}, \omega_2)),$$
allows us to conclude the same for the curve $\tau\in\mathbb{R}\mapsto\varphi(\tau, \mathbf{v}_{2}, \omega_2)$ and the line $F_{\omega_1}^{-1}(0)$.

Since for $i=1, 2$ the line $F_{\omega_i}^{-1}(0)$ is an asymptote to the curve $\tau\mapsto \varphi(\tau, \mathbf{v}_i, \omega_i)$ as $\tau\rightarrow+\infty$, the previous analysis imply the existence of real numbers $\tau_1, \tau_2, \rho_1, \rho_2$ satisfying 
$$\tau_2<0<\tau_1\hspace{.5cm}\mbox{ and }\hspace{.5cm} \rho_1<0<\rho_2$$
and 
$$\varphi(\tau_1, \mathbf{v}_{1}, \omega_1)=\varphi(\tau_2, \mathbf{v}_{2}, \omega_2)\hspace{.5cm}\mbox{ and }\hspace{.5cm}\varphi(\rho_1, \mathbf{v}_{1}, \omega_1)=\varphi(\rho_2, \mathbf{v}_{2}, \omega_2).$$
 Therefore, the piecewise constant functions
$$\omega_{12}(\tau)=\left\{\begin{array}{cc}
   \omega_1,  &  \tau\in [0, \tau_1]\\
    \omega_2, &  \tau\in (\tau_1, \tau_1-\tau_2] 
\end{array}\right. \hspace{.5cm}\mbox{ and }\hspace{.5cm}\omega_{21}(\tau)=\left\{\begin{array}{cc}
   \omega_2,  &  \tau\in [0, \rho_2]\\
    \omega_1, &  \tau\in (\rho_2, \rho_2-\rho_1] 
\end{array}\right.$$
satisfies,
\begin{align*}
\varphi(\tau_1-\tau_2, \mathbf{v}_{1}, \omega_{12})&=\varphi(-\tau_2, \varphi (\tau_1, \mathbf{v}_{1}, \omega_1), \omega_2)=\mathbf{v}_{2}  \\
\varphi(\rho_2-\rho_1, \mathbf{v}_{2}, \omega_{21})&=\varphi(-\rho_1, \varphi(\rho_2, \mathbf{v}_{2}, \omega_2), \omega_1)=\mathbf{v}_{1},
\end{align*}
which assures the existence of the periodic orbit between $\mathbf{v}_{1}$ and $\mathbf{v}_{2}$ as stated (see Figure \ref{fig3}).

\begin{figure}[H] 
	\centering
	\includegraphics[scale=.38]{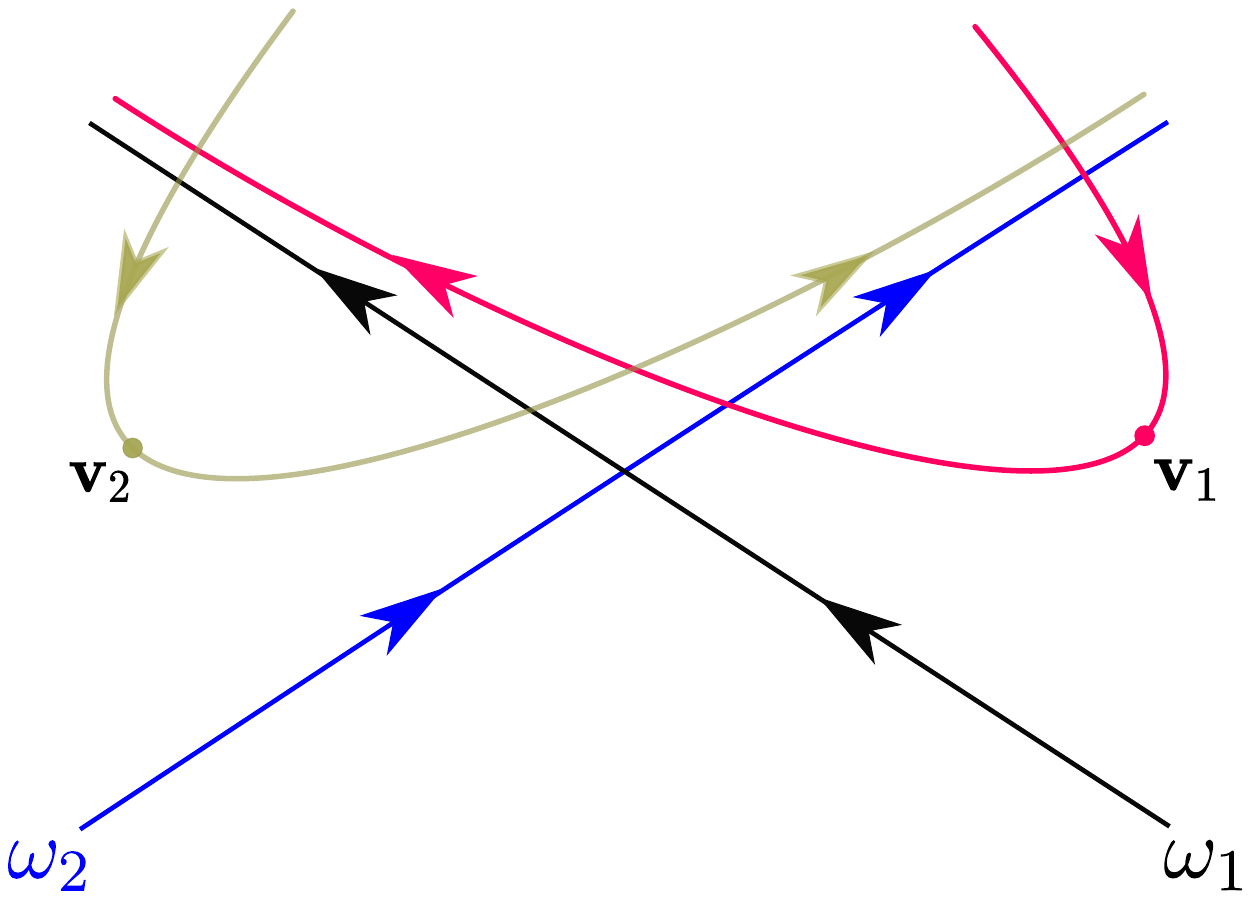}
 \vspace*{-9mm}
	\caption{A periodic orbit 
 through $\mathbf{v}_{1}$ and $\mathbf{v}_{2}$}
	\label{fig3}
\end{figure}
\end{proof}

 \bigskip
 
 We are now in a position to describe the only control set of $\Sigma_{1, 0}$.
 
 \begin{theorem}
     The system $\Sigma_{1, 0}$ admits a unique control set $\CC$ whose closure satisfies:
     \begin{itemize}
         \item[(1)] $a=0$ and $\overline{\CC}=\R\times -\displaystyle{\frac{c}{\lambda}\Omega}$;
         
         \item[(2)] $a>0$ and 
         $\overline{\CC}= \R^2\setminus \left(\CC^+\cup\CC^-\right).$
     \end{itemize}
     In both cases, $\CC$ is closed when $\lambda<0$ and open if $\lambda>0$
     
 \end{theorem}
 
 \begin{proof}
     (1) Since, by \cite[Lemma 5]{AEO}, the subset $\displaystyle{-\frac{c}{\lambda}\Omega}$ is the unique control set of the control system
$$\dot{t}=\lambda t+c\omega, \hspace{1cm}\omega\in\Omega,$$
by Proposition \ref{conjugation}, it is enough to show that the fiber $\R\times\{0\}$ is controllable. For any $\mathbf{v}_0=(s_0, 0), \mathbf{v}_1=(s_1, 0)\in\R\times\{0\}$ there exists, by the continuous dependence of the initial conditions, $\omega, \omega_0, \omega_1\in\inner\Omega$, $\tau_0, \tau_1>0$ and $\tilde{s}_0, \tilde{s}_1\in\R$ satisfying
$$\varphi_2(\tau_0, \mathbf{v}_0, \omega_0)=\left(\tilde{s}_0, -\frac{c}{\lambda}\omega, \right),\hspace{.5cm}\varphi_2\left(\tau_1, \left(\tilde{s}_1, -\frac{c}{\lambda}\omega\right), \omega_1\right)=\mathbf{v}_1, \hspace{.5cm}\mbox{ and }\hspace{.5cm}(\tilde{s}_1-\tilde{s}_0)\omega>0.$$
 A trajectory of $\Sigma_{1, 0}$ connecting $\mathbf{v}_0$ to $\mathbf{v}_1$ is constructed by concatenation as:

\begin{enumerate}
    \item[(i)] With time, $\tau_0>0$ and control $\omega_0\in \Omega$ go from $\mathbf{v}_0$ to $\left(\tilde{s}_0, -\displaystyle{\frac{c}{\lambda}}\omega\right)$;

    \item[(ii)] With time $\tau=\frac{\tilde{s}_1-\tilde{s}_0}{\omega}>0$ and control $\omega$ go from $\left(\tilde{s}_0, -\displaystyle{\frac{c}{\lambda}}\omega\right)$ to 
    $$\varphi\left(\tau, \left(\tilde{s}_0, -\displaystyle{\frac{c}{\lambda}}\omega\right), \omega\right)=\left(\tilde{s}_0+\tau\omega, -\frac{c}{\lambda}\omega\right)=\left(\tilde{s}_1, -\frac{c}{\lambda}\omega\right);$$

    \item[(iii)] With time $\tau_1>0$ and control $\omega_1\in \Omega$ go from $\left(\tilde{s}_1, -\displaystyle{\frac{c}{\lambda}}\omega\right)$ to $\mathbf{v}_1$.
\end{enumerate}
 
By the arbitrariness of $\mathbf{v}_0$ and $\mathbf{v}_1$ we get the controllability of $\R\times\{0\}$, showing (1).

\bigskip

\hspace{1.1cm}(2) By Lemma \ref{lema}, if $\lambda<0$, the set $\CC^+\cup\CC^-$ is invariant in negative time and hence, $\R^2\setminus(\CC^+\cup\CC^-)$ is positively invariant. Therefore, if we show that controllability holds in $\inner(\R^2\setminus(\CC^+\cup\CC^-))$ the result follows.

However, 
$$\inner(\R^2\setminus(\CC^+\cup\CC^-))=\CC_1\cap\CC_2\cup\CC(\omega^-, \omega^+),$$ 
where
$$\CC_1:=\Bigl\{(s, t)\in\R^2~:~\exists\tau\geq 0 \mbox{ and }\omega\in\{\omega^-, \omega^+\} \mbox{ with } s=\varphi_1(\tau, \mathbf{v}_a, \omega)\mbox{ and }t<\varphi_2(\tau, \mathbf{v}_a, \omega)\Bigr\},$$
and 
$$\CC_2:=\{(s, t)\in\R^2~:~ F_{\omega^+}(s, t)\geq 0\;\mbox{ and }\;F_{\omega^-}(s, t)\geq 0\}.$$
Since by Lemma \ref{lema} controllability holds in $\CC(\omega^-, \omega^+)$, it is enough to show that, for any point in $\CC_1\cap\CC_2$, there exists a trajectory starting and finishing in $\CC(\omega^-, \omega^+)$.

Let us first note that the lines $F_{\omega}^{-1}(0)$ can be parametrized by the curve
$$s\in\R\mapsto \mathbf{v}_{\omega}(s)=\left(s, -\frac{\omega}{\lambda^2}(\lambda(as+c)+a\omega)\right)\in\R^2,$$
and hence, for $\omega_1, \omega_2\in\Omega$, we get that 
$$F_{\omega_1}(\mathbf{v}_{\omega_2}(s))=(\omega_1-\omega_2)(\lambda(as+c)+a(\omega_1+\omega_2)).$$
As a consequence, for any $\omega\in\Omega$, there exists $s_{\omega}>0$ such that 
    $$\mathbf{v}_{\omega}(s)\in\CC(\omega^-, \omega^+),\;\;\;\mbox{ for }\;\;\;|s|\geq s_{\omega}.$$
    Since the lines $F_{\omega}^{-1}(0)$ are asymptotes to the curve $\tau\in\R\mapsto\varphi(\tau, \mathbf{v}, \omega)$, we get that 
     $$\varphi(\tau_0, \mathbf{v}, \omega)\in\CC(\omega^-, \omega^+), \;\mbox{ for some }\tau_0>0,$$
     showing that we can reach $\CC(\omega^-, \omega^+)$ from any point in $\CC_1\cap\CC_2$ (see Figure \ref{fig4}).
     
    Next, let us construct, for a given $\mathbf{v}\in\CC_1\cap\CC_2$, an orbit starting in $\CC(\omega^-, \omega^+)$ and finishing on $\mathbf{v}$. 
    Since solutions of ODEs are parallel or coincident, there exists $\tau_1\geq 0$ and $t<0$ such that 
    $$\varphi(\tau, (-c/a, t), \omega^i)=\mathbf{v}, \;\;\;\mbox{ where }\;\;\;\omega^i=\left\{\begin{array}{ll}
        \omega^- ,&  as+c\leq 0 \\
        \omega^+ ,&  as+c> 0
    \end{array}\right.$$
    On the other hand, the fact that
    $$F_{\omega}^{-1}(0)\cap\{s=-c/a\}=\left\{\left(-c/a,-\frac{\omega^2}{\lambda^2}\right)\right\},$$
    implies the existence of $\omega_*\in\Omega$ such that 
    $$\mathbf{v}_{1}=F_{\omega_*}^{-1}(0)\cap \varphi([0, \tau_1], (-c/a, t), \omega^i)\}, \;\;\;\mbox{ with }\;\;\;\omega_*\omega^i<0.$$
    By the previous discussion, we can take $\mathbf{v}_{2}\in F_{\omega_*}^{-1}(0)\cap\CC(\omega^-, \omega^+)$ satisfying,  
    $$\mathbf{v}_{1}=\varphi(\tau_2, \mathbf{v}_{2}, \omega_*), \hspace{.5cm}\mbox{ for some }\tau_2>0.$$
    Therefore, 
    $$\varphi(\tau_1-\rho, \varphi(\tau_2, \mathbf{v}_{2}, \omega_*), \omega^i)=\mathbf{v},\;\;\;\mbox{ where }\;\;\;\mathbf{v}_{1}=\varphi(\rho, (-c/a, t), \omega^i), \;\;\;\rho\in [0, \tau_1],$$
    showing that we can reach any $\mathbf{v}\in\CC_1\cap\CC_2$ from a point in $\CC(\omega^-, \omega^+)$ (Figure \ref{fig4}) and concluding the proof.

\begin{figure}[H] 
	\centering
	\includegraphics[scale=.38]{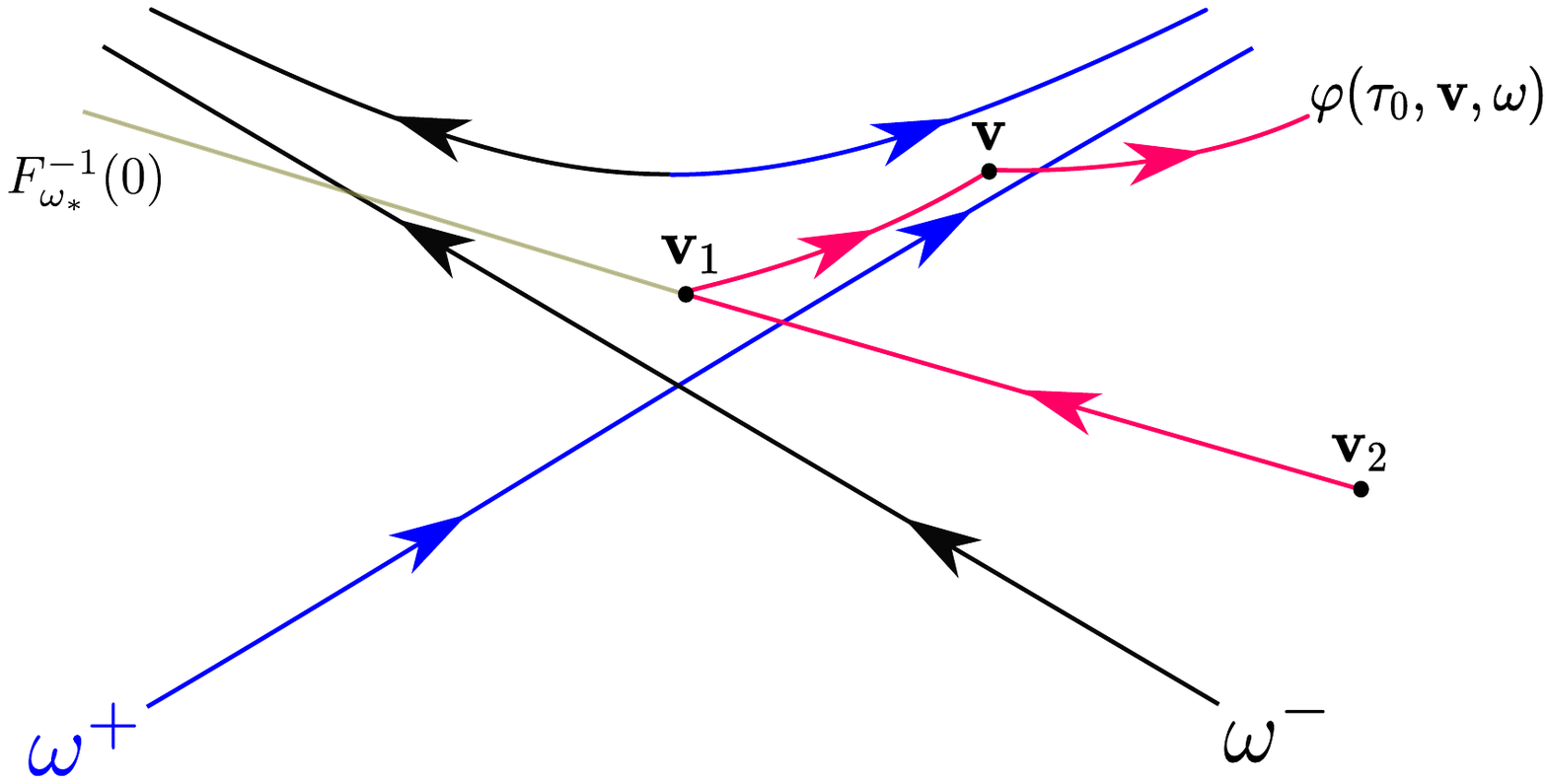}
 \vspace*{-9mm}
	\caption{Orbit through a point  $\mathbf{v}\in \CC_1\cap\CC_2$ starting and ending in $\CC(\omega^-, \omega^+)$.}
	\label{fig4}
\end{figure}
 
\end{proof}

Finally, we explain the last subcase of the case $\alpha=0$.
\subsubsection{The subcase $\beta\neq 0, \lambda+\beta=0$:}

In this case, under the LARC, the diffeomorphism
  $$f:\mathbb{R}^2\rightarrow \R^2,\hspace{1cm}f(s, t)=\left(-\frac{\beta s}{b}, -\frac{\beta}{b}s+\frac{\gamma}{b}t\right),$$
  conjugates our initial system to a system of the form
   \begin{flalign*}
	 	&&\left(\Sigma_{1, 0}\right):\left\{
	 	\begin{array}{l}
	 		\dot{s}=\beta(s-\omega)\\ \dot{t}=(a\omega+\gamma)s\end{array}\right. &&
	 \end{flalign*}	
where $\omega  \in \Omega$. It satisfies LARC if and only if $a^2+c^2\neq 0$. \\
As previously, let us assume w.l.o.g. is that $\beta<0$. In this case, the fact that $\Omega$ is the unique control set of the control system (see, for instance, \cite[Lemma 5]{AEO})
$$\dot{s}=\beta(s-\omega), \hspace{1cm}\omega\in\Omega,$$
implies that any control set of $\Sigma_{1, 0}$ has to be inside the subset $\Omega\times\R$. Moreover, since the first component of the solutions of $\Sigma_{1, 0}$ for constant control is given by
 $$\varphi_1(\tau, \mathbf{v}_0, \omega)=\rme^{\beta \tau}(s_0-\omega)+\omega,\hspace{.5cm}\mathbf{v}_0:=(s_0, t_0)$$
 it holds that 
$$ \varphi_{\tau, \omega}(\Omega\times\R)\subset \Omega\times\R, \hspace{.5cm} \forall \tau\geq 0.$$

\begin{theorem}
\label{equilibria}
    For the previous system, it holds that:
\begin{itemize}
    \item[(1)]$\gamma\neq 0$ and $\CC=\Omega\times\R$ is the only control set of $\Sigma_{1, 0}$;

    \item[(2)] $\gamma=0$ and $\{(0, t)\}$ are distinct control sets of $\Sigma_{1, 0}$.
\end{itemize}
    
\end{theorem}

\begin{proof} (1) By Proposition \ref{conjugation}, we only have to show that the fiber $\{0\}\times\R$ is controllable. Moreover, let us assume, w.l.o.g., that $\gamma>0$ since the other possibility is analogous. Let then $\mathbf{v}_0=(0, t_0)$ and $\mathbf{v}_1=(0, t_1)$ satisfy $t_0<t_1$.

Then, 
$$\tau=\frac{t_1-t_0}{\gamma}\hspace{.5cm}\implies\hspace{.5cm}\varphi(\tau, \mathbf{v}_0, 0)=(0, t_0+\gamma\tau)=(0, t_1)=\mathbf{v}_1.$$

On the other hand, by the continuous dependence of the initial conditions, there exists $\omega\in\inner\Omega$ such that 
$(a\omega+\gamma)\omega<0$ and 
$$\varphi_1(\tau_1, \mathbf{v}_1, \omega_1)=(\omega, \tilde{t}_1)\hspace{.5cm}\mbox{ and }\hspace{.5cm}\varphi_1(\tau_2, (\omega, \tilde{t}_0), \omega_2)=\mathbf{v}_0, \hspace{.5cm}\mbox{ with }\hspace{.5cm}\tilde{t}_0<\tilde{t}_1,$$
for some $\omega_1, \omega_2\in\Omega$ and $\tau_1, \tau_2>0$. A trajectory of $\Sigma_{1, 0}$ starting in $\mathbf{v}_1$ and finishing in $\mathbf{v}_0$ is constructed by concatenation as:

\begin{enumerate}
    \item[(i)] With time, $\tau_1>0$ and control $\omega_1\in \Omega$ go from $\mathbf{v}_1$ to $(0, \tilde{t}_1)$;

    \item[(ii)] With time $\tau=\frac{\tilde{t}_0-\tilde{t}_1}{(a\omega+\gamma)\omega}>0$ and control $\omega$ go from $(0, \tilde{t}_1)$ to 
    $$\varphi(\tau, (0, \tilde{t}_1), \omega)=(\omega, \tilde{t}_1+(a\omega+\gamma)\omega\tau)=(\omega, \tilde{t}_0);$$

    \item[(iii)] With time, $\tau_2>0$ and control $\omega_2\in \Omega$ go from $(\omega, \tilde{t}_0)$ to $\mathbf{v}_0$.
\end{enumerate}

(2) If $\gamma=0$, then $a\neq 0$, and we will assume w.l.o.g. that $a>0$, since the other possibility is analogous. Define the function 
$$F:\Omega \times\R\rightarrow\R, \hspace{1cm} F(s, t):=t+\frac{a\sigma}{\beta}s-\frac{a\sigma^2}{\beta}\ln(s+\sigma),$$
where $\sigma>0$ satisfies $\sigma>|\omega|$ for all $\omega\in\Omega$.

Let $\omega\in\Omega$ and $\mathbf{v}_0:=(s_0, t_0)\in\inner\Omega\times\R$ and write $\varphi(\tau, \mathbf{v}_0, \omega)=(s, t)$. Then,
$$\frac{d}{d\tau}F(\varphi(\tau, \mathbf{v}_0, \omega))=\dot{t}+\frac{a\sigma}{\beta}\dot{s}-\frac{a\sigma^2}{\beta}\frac{\dot{s}}{s+\sigma}=a\omega s+\frac{a\sigma}{\beta}\beta(s-\omega)-\frac{a\sigma^2}{\beta}\frac{\beta(s-\omega)}{s+\sigma}$$
$$\frac{a\omega s(s+\sigma)+a\sigma(s-\omega)(s+\sigma)-a\sigma^2(s-\omega)}{s+\sigma}=\frac{a(\omega+\sigma)}{s+\sigma}s^2.$$
By concatenation, we easily conclude that,
$$F(\varphi(\tau, \mathbf{v}_0, {\bf \omega}))>F(\mathbf{v}_0), \hspace{.5cm}\mbox{ if } \hspace{.5cm}\omega\not\equiv 0 \hspace{.5cm}\mbox{ or }\hspace{.5cm}s_0\neq 0,$$
implying that any control set of $\Sigma_{1, 0}$ has to be contained in one of the curves $F^{-1}(c)$ and have the first component equal to zero. Since both properties happen only at exactly one point, we conclude that $\{(0, t)\}$ are the only control sets of $\Sigma_{1, 0}$ inside $\Omega\times\R$, concluding the result.

\end{proof}

\subsection{Controllability in the case $\alpha\neq 0$.}

Let us now analyze the case where $\alpha\neq 0$. By Proposition \ref{oneinput1D}, such conditions imply that $\gamma=0$. Moreover, if $\lambda\neq \beta$, the map
$$f:\mathbb{R}^2\rightarrow \R^2,\hspace{1cm}f(s, t)=\left(s, t+\frac{\alpha}{2(\lambda-\beta)}s^2\right),$$
is a diffeomorphism that conjugates our initial LCS to the LCS 
 \begin{flalign*}
	 	&&\left(\Sigma_{1, 0}\right):\left\{
	 	\begin{array}{l}
	 		\dot{s}=\beta s+\omega b\\
	 		\dot{t}=(\lambda+\beta)t+\omega(c+as)
	 	\end{array}\right.,  &&
	 \end{flalign*}	
  that has no quadratic term. As a consequence, we only have to analyze singular LCSs where $\alpha\neq 0$ and $\lambda=\beta$, since the other possibilities were studied in the previous sections. However, the only singular LCS where $\lambda=\beta$ and $\alpha\neq 0$ is given by
  \begin{flalign*}
	 	&&\left(\Sigma_{1, 0}\right):\left\{
	 	\begin{array}{l}
	 		\dot{s}=\omega b\\
	 		\dot{t}=\alpha s^2+\omega(c+as)
	 	\end{array}\right. &&
	 \end{flalign*}	
  and for such a system, we have the following.

\begin{proposition}
\label{alfanozero}
    If a singular LCS $\Sigma_{1, 0}$ is such that $\alpha\neq 0$ and $\lambda=\beta=0$, then it admits only one-point control sets given by the singularities of the drift.
\end{proposition}

\begin{proof} The diffeomorphism 
$$f:\R^2\rightarrow\R^2, \hspace{1cm}f(s, t)=\left(s, t-\frac{2c}{b}s\right),$$
conjugates the previous LCS to the system 
   \begin{flalign*}
	 	&&\left(\Sigma_{1, 0}\right):\left\{
	 	\begin{array}{l}
	 		\dot{s}=\omega b\\
	 		\dot{t}=\alpha s^2+2a\omega s
	 	\end{array}\right.
   &&
	 \end{flalign*}	  

For this system, let us define the function 
$$G:\R^2\rightarrow\R, \hspace{1cm} G(s, t)=6\sigma t-6\frac{a}{b} \sigma s^2+2s^3,$$
where $\sigma\in\R$ satisfies $\sigma\alpha+\omega b>0$ for all $\omega\in \Omega$. Then, for any $\omega\in\Omega$ the function
$$\tau\mapsto G(\varphi(\tau, \mathbf{v}_0,\omega)), \hspace{1cm}\mathbf{v}_0=(s_0, t_0),$$
satisfies
$$\hspace{-2cm}\frac{d}{d\tau}G(\varphi(\tau, \mathbf{v}_0,\omega))=6\sigma\dot{t}-12\frac{a}{b} \sigma s\dot{s}+6s^2\dot{s}$$
$$\hspace{2cm}=6\left(\sigma(\alpha s^2+2a\omega s)-2\frac{a}{b}\sigma\omega b s+ \omega b s^2\right)=6(\sigma\alpha+\omega b)s^2,$$
which, as in Theorem \ref{equilibria} allows us to conclude that the only control sets are the singularities of the drift, that is,  $\{(0, t)\}, t\in\R$ are the control sets of $\Sigma_{1, 0}$.
\end{proof}


\begin{thebibliography}{9}

\bibitem{Ayala} Ayala V, Tirao J. Linear Control Systems on Lie Groups and Controllability. Proceedings of Symposia in Pure Mathematics, 1999;64.

\bibitem{VAAA} Ayala V, Da Silva A, Torreblanca M. Linear control systems on the homogeneous spaces of the 2D Lie group. J Differ Equ. 2022;314:850-870.

\bibitem{VA1} Ayala V, Da Silva A. The control set of a linear control system on the two dimensional Lie group. J Differ Equ. 2020;268(11):6683-6701.

\bibitem{VAP} Ayala V, Da Silva A, Jouan P, Zsigmond G. Control sets of linear systems on semi-simple Lie groups. J Differ Equ. 2020;269(5):449-466.

\bibitem{VA2} Ayala V, Da Silva A. On the characterization of the controllability property for linear control systems on nonnilpotent, solvable three-dimensional Lie groups. J Differ Equ. 2019;266(12):8233-8257.

\bibitem{VA3} Ayala V, Da Silva A. Controllability of linear control systems on Lie groups with semisimple finite center. SIAM J Control Optim. 2017;55(2):1332-1343.

\bibitem{AEO} Da Silva, A., Kizil, E., Duman, O. Linear Control Systems on Homogeneous Spaces of the Heisenberg Group. J Dyn Control Syst 29, 2065–2086 (2023)

\bibitem{FW} Colonius F, Kliemann W. The Dynamics of Control. Birkhäuser: Boston; 2000.

\bibitem{Jouan} Jouan P. Equivalence of control systems with linear systems on Lie groups and homogeneous spaces. ESAIM: Control Optim Calc Var. 2010;16:956-973.

\bibitem{Sussmann} Jurdjevic V, Sussmann HJ. Control Systems on Lie Groups. J Differ Equ. 1972;12:313-329.

\bibitem{Markus} Markus L. Controllability of multi-trajectories on Lie groups, in: Proceedings of Dynamical Systems and Turbulence, Warwick, in: Lecture Notes in Mathematics. 1980;898:250-265.

\bibitem{GL} Leitmann G, Optimization Techniques with Application to Aerospace Systems, Academic Press Inc., 1962.
\bibitem{KS} Shell K, Applications of Pontryagin’s Maximum Principle to Economics, Mathematical Systems Theory and Economics I and II, Lecture Notes in Operations Research and Mathematical Economics, v. 11 (1968) 241-292.

\end{thebibliography}
\end{document}